\documentclass[11pt, a4paper, oneside, dvipdfmx]{amsart} % reqno

\usepackage{amsmath, amssymb, amsthm, amsrefs}
\usepackage{mathrsfs}
\usepackage{ascmac}
\usepackage{graphicx}
\usepackage{xcolor}
\usepackage{bm}
\usepackage[italicdiff]{physics}
\usepackage{mathtools}

\usepackage{hyperref}
\hypersetup{
	colorlinks=true, %set true if you want colored links
	linktoc= page,   %set to all if you want both sections and subsections linked
	citecolor=blue
}

\allowdisplaybreaks[4]
\numberwithin{equation}{section}

\def\({\left(}
\def\){\right)}
\def\le{\leqslant}
\def\ge{\geqslant}
\def \ds{\displaystyle}

\def \R{\mathbb{R}}
\def \C{\mathbb{C}}

\def \N{\mathbb{N}}

\theoremstyle{plain}
\newtheorem{theorem}{Theorem}[section]
\newtheorem{definition}[theorem]{Definition}
\newtheorem{lemma}[theorem]{Lemma}

\newtheorem{proposition}[theorem]{Proposition}
\theoremstyle{remark}
\newtheorem{remark}[theorem]{Remark}

\begin{document}
\title[]{Refinement of the $L^{2}$-decay estimate of solutions to nonlinear Schr\"odinger equations with attractive-dissipative nonlinearity}
\author[N. Kita]{Naoyasu Kita}
\address[]{Faculty of Advanced Science and Technology, Kumamoto University, Kumamoto, 860-8555, Japan}
\email{nkita@kumamoto-u.ac.jp}
\author[H. Miyazaki]{Hayato Miyazaki}
\address[]{Teacher Training Courses, Faculty of Education, Kagawa University, Takamatsu, Kagawa 760-8522, Japan}
\email{miyazaki.hayato@kagawa-u.ac.jp}
\author[T. Sato]{Takuya Sato}
\address[]{Faculty of Advanced Science and Technology, Kumamoto University, Kumamoto, 860-8555, Japan}
\email{satotakuya@kumamoto-u.ac.jp}
\date{}
\maketitle

\begin{abstract}
This paper is concerned with the $L^{2}$-decay estimate of solutions to nonlinear dissipative Schr\"odinger equations with power-type nonlinearity of the order $p$.
It is known that the sign of the real part of the dissipation coefficient affects the long-time behavior of solutions,
when neither size restriction on the initial data nor strong dissipative condition is imposed.
In that case, if the sign is negative, then
Gerelmaa, the first and third author \cite{GKS-pre} obtained the $L^{2}$-decay estimate under the restriction $p \le 1+1/d$.
In this paper, we relax the restriction to $p \le 1+4/(3d)$ by refining an energy-type estimate.
Furthermore, when $p < 1+ 4/(3d)$, using an iteration argument, the best available decay rate is established, as given by Hayashi, Li and Naumkin \cite{HLN16}.
\end{abstract}

\section{Introduction} \label{sec:1}
\subsection{Background}
In this paper, we consider a nonlinear Schr\"odinger equation
\begin{align}
\begin{cases}
	\ds i \partial_t u  +  \frac{1}{2} \Delta u  = \lambda |u|^{p-1} u, & (t,x) \in (0, \infty) \times \R^{d}, \\
	u(0, x) = u_{0}(x), & x \in \R^{d},
\end{cases}
	\label{nls} \tag{NLS}
\end{align}
where $d \in \N$, $u = u(t,x)$ is a $\C$-valued unknown function, $u_{0} = u_{0}(x)$ is a given function,
$p>1$, and $\lambda \in \C$.
One imposes a \textit{dissipative condition}
\begin{align*}
	\Re \lambda \in \R, \quad
	\Im \lambda < 0.
\end{align*}
In the case $\Re \lambda < 0$, the condition is called an \textit{attractive-dissipative condition}; otherwise it is said to be a \textit{repulsive-dissipative condition}.
The equation \eqref{nls} appears as a relevant model in various physical phenomena, such as nonlinear optics, Bose-Einstein condensates and so on.
For instance, in the optical fiber engineering,
$\Re \lambda$ corresponds to the magnitude of the nonlinear Kerr effect, and $\Im \lambda$ implies the magnitude
of dissipation due to nonlinear Ohm's law (see e.g. \cite{GA19}).
The dissipative condition is firstly introduced by Shimomura \cite{S-CPDE-2006} on the study of the long-time behavior of solutions to \eqref{nls}.
The condition implies that the $L^2$-norm of the solution monotonically decreases as $t \rightarrow \infty$, because
the solution satisfies
\begin{align}
	\norm{u(t)}_{2}^{2} = \norm{u_{0}}_{2}^{2} + \Im \lambda \int_{0}^{t} \norm{u(s)}_{p+1}^{p+1}\, ds
	\label{ide:1}
\end{align}
for any $t \ge 0$.
The aim of this paper is to investigate the time decay rate of the $L^{2}$-norm of solutions to \eqref{nls}
without the size restriction on the initial data in any dimensions.

For the case $\lambda \in \R \setminus \{0\}$, the equation \eqref{nls} possesses a Hamiltonian structure.
In particular, the $L^{2}$-norm is conserved for $t$.
In that case, if $p \ge 1+ 2/d$,
then it is known that \eqref{nls} admits a global solution that decays in time on the order $-d/2$ in $L^{\infty}$,
provided that the data $u_{0}$ is small in an appropriate weighted Sobolev space (e.g., \cite{CW92, GOV94, HN-AJM-1998}).
The decay rate is the same as that of linear solutions.
In contrast, to the best of our knowledge, when $p< 1+ 2/d$, there are no results regarding
the long-time behavior of solutions to \eqref{nls} under $\lambda \in \R \setminus \{0\}$.

When $\Im \lambda \neq 0$, the long-time behavior of the solution exhibits different aspects.
As the identity \eqref{ide:1} shows,
if $\Im \lambda >0$, then the $L^{2}$-solution blows up in finite time (e.g., \cite{K20a, CMZ19}).
In the case $\Im \lambda < 0$, we shall first review some known results under small data.
For $d \le 3$, when $p$ is sufficiently close to $1+2/d$ with $p \le 1+2/d$, the long-time behavior of the solutions
to \eqref{nls} was identified by Shimomura \cite{S-CPDE-2006}, and the first author and Shimomura \cite{KS-JDE-2007}.
Specifically, they determined the decay rate of the solution as
\[
	\norm{u(t)}_{\infty} \lesssim
	\begin{cases}
		(t \log t)^{-\frac{d}{2}} & \text{if}\ p=1+2/d, \\
		t^{- \frac{1}{p-1}} & \text{if}\ p<1+2/d,
	\end{cases}
\]
and proved that the solution possesses $\norm{u(t)}_{2} \rightarrow 0$ as $t \rightarrow \infty$.
In the case $p > 1 + 2/d$,
the third author \cite{S-AM-2020} showed that the $L^{2}$-norm of the solution does not decay in time and has a lower bound.
These results demonstrate that the long-time behavior of solutions to \eqref{nls} under the dissipative condition
is different from that under the non-dissipative condition $\lambda \in \R \setminus \{0\}$.

The first attempt to study the time decay of the $L^{2}$-norm of solutions without the size restriction on the data
was conducted by the first author and Shimomura \cite{KS-JMSJ-2009}.
They showed that in $d = 1$, under the strong dispersive condition
\begin{align}
      (p-1) |\lambda| \le (p+1) |\Im \lambda|, \quad \Im \lambda < 0,
      \label{sdc}
\end{align}
when $p \le 1 + 2/d$ is close to $1 + 2/d$,
the solution satisfies $\|u(t)\|_{2} \rightarrow 0$ as $t \rightarrow \infty$.
The condition \eqref{sdc} was firstly introduced by Liskevich and Perelmuter \cite{LP95} to study the analyticity
of some semigroup (cf. Okazawa and Yokota \cite{OY02}).
Eventually, assuming the condition \eqref{sdc}, Hayashi, Li and Naumkin \cite{HLN16} identified the time decay rate
of the $L^{2}$ norm of the solution under the initial data $u_{0} \in \Sigma \coloneqq H^{1} \cap \mathcal{F} H^{1}$ as follows:
\begin{align}
	\norm{u(t)}_{2} \lesssim
	\begin{cases}
		(\log t)^{-\frac{d}{d+2}} & \text{if}\ p=1+2/d, \\
		t^{- \frac{d}{d+2}\(\frac{2}{d(p-1)}-1\)} & \text{if}\ p<1+2/d.
	\end{cases}
	\label{hln:rate}
\end{align}
In \cite{KS-AA-2021, KS22b},
the optimality of the decay rate of \eqref{hln:rate} for $d=1$ is discussed under $u_{0} \in \Sigma$, and
they suggested that the factor $d/(d+2)$ in the decay rate of \eqref{hln:rate} could potentially be refined to $d/2$.
However, the best available rate for any $\Sigma$-solutions is in \eqref{hln:rate}.
Another direction for the time decay of the $L^{2}$-norm of solutions is to investigate the relation
between the regularity of solutions and the decay rate.
The direction is treated in Ogawa and the third author \cite{OS-NoDEA-2020} (see also \cite{S-AM-2020, S22a, S-AM-2020}).
It is also indicated that the time decay rate depends on the weighted conditions of the initial data.
In fact, for $d \ge 1$, Cazenave and Naumkin \cite{CN18}, and Cazenave, Han and Naumkin \cite{CHN-NA-2021}
showed that the decay rate of the $L^{2}$ norm of the solution depends on the weight $N$,
under specific initial data $u_{0}$ involved the quadratic phase oscillation $e^{ib |x|^{2}}$ with $b \ll 1$
satisfying $\inf \ev{x}^{N}|u_{0}(x)| > 0$.

The above results on the $L^{2}$-decay of the solution, without imposing any size restriction on the data,
require the strong dissipative condition \eqref{sdc}, except for \cite{CN18} and \cite{CHN-NA-2021}.
In this paper, when $p \le 1 + 2/d$, we identify the time decay rate of the $L^{2}$ norm of the solution
without condition \eqref{sdc} and the size restriction on the initial data $u_{0}$ merely in $\Sigma$,
which is a natural space as the behavior of the solution is considered.
In our setting, the sign of $\Re \lambda$ affects the behavior of the solution.
In fact, when $p \le 1+2/d$, Gerelmaa, the first and third author \cite{GKS24} discussed the $L^{2}$-decay estimate
of the solution to \eqref{nls} under the repulsive dissipative condition $\Re \lambda > 0$ and $\Im \lambda < 0$,
and showed that the decay rate coincides with \eqref{hln:rate}.
On the other hand, under the attractive dissipative condition $\Re \lambda < 0$ and $\Im \lambda < 0$, Gerelmaa,
the first and third author \cite{GKS-pre} obtained the following:
\begin{theorem}[\cite{GKS-pre}] \label{prop:32}
Let $1<p \le 1+1/d$.
Assume $\Re \lambda < 0$ and $\Im \lambda < 0$.
Given $u_{0} \in \Sigma$ with $ \norm{u_{0}}_{\ \Sigma} \neq 0$,
\eqref{nls} admits a unique global solution $u \in C([0, \infty) ; \Sigma)$.
If $p= 1+1/d$, then it holds that
\[
	\norm{ u(t)}_{2} \lesssim (\log (1+t))^{- \frac{d}{(d+2)(p-1)}}
\]
for any $t \ge 1$.
When $1<p< 1+ 1/d$, the estimate
\[
	\norm{ u(t)}_{2} \lesssim (1+t)^{- \frac{d}{d+2}\( \frac{2}{d(p-1)} - \frac{3-d(p-1)}{2-d(p-1)} \)}
\]
is valid for any $t \ge 0$.
\end{theorem}
In Theorem \ref{prop:32}, the decay rate of the $L^2$-norm is slower than \eqref{hln:rate}.
Further, the more restriction $p \le 1+1/d$ is required.
In this paper, we relax the restriction to $p \le 1+4/(3d)$,
and establish the same decay rate as \eqref{hln:rate} when $p < 1+4/(3d)$.

\subsection{Main result}

Based on the Duhamel principle, we give the definition of solutions to \eqref{nls} as follows:
\begin{definition}[Solution] \label{def:sol}
Let $I \subset \R$ be an interval and fix $t_{0} \in I$.
We say a function $ u \colon I \times \R^d \to \C^{n}$ is a solution to \eqref{nls} on $I$ if $u \in C(I; L^2(\R^d))$
satisfies
\[
	u(t) = U(t- t_{0}) u(t_{0}) - i \int_{t_{0}}^{t} U(t- s) F( u(s))\, ds
\]
in $L^2(\R^d)$ for any $t \in I$, where $U(t) = e^{\frac{it \Delta}{2}}$
and $F(u)=\lambda |u|^{p-1}u$.
\end{definition}

The main result in this paper
is as follows:

\begin{theorem} \label{thm:1}
Let $1<p \le 1+4/(3d)$.
Assume $\Re \lambda <0$ and $\Im \lambda < 0$.
Take $u_{0} \in \Sigma$ with $ \norm{u_{0}}_{\ \Sigma} \neq 0$.
Then \eqref{nls} admits a unique global solution $u \in C([0, \infty) ; \Sigma)$.
Moreover, if $p = 1 + 4/(3d)$, then it holds that
\[
	\norm{ u(t)}_{2} \lesssim (\log (1+t))^{- \frac{3d}{2(d+2)}}
\]
for any $t \ge 1$.
When $1<p < 1+ 4/(3d)$, the estimate
\[
	\norm{ u(t)}_{2} \lesssim (1+t)^{- \frac{d}{d+2}\(\frac{2}{d(p-1)}-1\)}
\]
is valid for any $t \ge 0$.
\end{theorem}

\begin{remark}
The restriction $p \le 1+4/(3d)$ is imposed in Proposition \ref{prop:31},
due to the uniform growth estimate $\norm{\nabla u(t)}_{2} \lesssim t^{1/2}$ given in Proposition \ref{prop:22}.
If $1+4/(3d) < p \le 1+2/d$, it remains an open problem to establish the decay estimate of the $L^{2}$-norm of the solution under the attractive-dissipative condition $\Re \lambda < 0$ and $\Im \lambda < 0$,
without imposing any size restriction on the initial data.
\end{remark}

\begin{remark}
Our proof is also applicable to the case of the repulsive-dissipative condition $\Re \lambda > 0$ and $\Im \lambda < 0$.
\end{remark}

\subsection{Strategy of the proof}
The strategy of the proof of Theorem \ref{thm:1} relies on the argument in \cites{K20a, GKS24}.
Applying the standard energy-type method to \eqref{nls}, H\"older's inequality leads to
\begin{align}
	\frac{d}{dt} \norm{u(t)}_{2}^{2}
	\lesssim - \norm{u(t)}_{2}^{\beta} \norm{u(t)}_{q}^{p+1 - \beta}
	\label{stra:1}
\end{align}
for some $\beta>p+1$, where $q =1$ if $d=1$ and $q = 2(d+1)/(d+2)$ if $d \ge 2$.
The key to obtain the desired decay estimate is to extract as much of the time decay factor as possible from the right-hand side of \eqref{stra:1}.
To this end, we apply the following scale-invariant inequality, which can be regarded as the dual of Gagliardo-Nirenberg's inequality:
\[
	\norm{u}_{q} \lesssim \norm{u}_{2}^{1- d \alpha} \norm{x u}_{2}^{ d \alpha}, \quad \alpha \coloneqq \frac{1}{q} - \frac{1}{2}.
\]
This inequality implies that it suffices to obtain a uniform-in-time estimate of $\norm{xu(t)}_{2}$, in order to get the desired $L^{2}$-decay estimate.
We remark that the norm $\norm{xu(t)}_{2}$ grows over time in general and
the growth rate depends on that of $\norm{\nabla u(t)}_{2}$. When $\Re \lambda \ge 0$, $\norm{\nabla u(t)}_{2}$
is uniformly bounded in time and hence $\norm{x u(t)}_{2} \lesssim t$ holds for large time.
The uniform estimate yields the best available rate in \eqref{hln:rate}.
In contrast,
when $\Re \lambda < 0$,
the norm $\norm{\nabla u(t)}_{2}$ may grow over time, leading to a higher growth for $\norm{xu(t)}_{2}$.
This, in turn, worsens the decay rate of the $L^{2}$-norm of the solution.

In Theorem \ref{prop:32} of \cite{GKS24}, the restriction $p \le 1+1/d$ arises from the use
of the uniform estimate $\norm{\nabla u(t)}_{2} \lesssim t^{\frac{1}{2-d(p-1)}}$ for large time.
We shall show the refined estimate $\norm{\nabla u(t)}_{2} \lesssim t^{1/2}$ and improve upon the restriction.
To get this refined estimate,
one develops the energy-type estimate
\begin{align}
	\frac{d}{dt} E(u(t)) \lesssim \norm{u(t)}_{2}^{\frac{2(4p-d(p-1))}{4-d(p-1)}}, \quad
	E( u(t)) \coloneqq \frac{1}{2} \norm{ \nabla u(t)}_{2}^{2} + \frac{ \lambda_{1}}{p+1} \norm{u(t)}_{p+1}^{p+1},
	\label{ene:int}
\end{align}
by picking up dissipative effect from the nonlinearity.
Further, in the case $p < 1+ 4/(3d)$,
combining the refined $L^{2}$-decay estimate with \eqref{ene:int} leads to improving
the growth rate of the uniform estimate for $\norm{\nabla u(t)}_{2}$, thereby refining the decay rate of $\norm{u(t)}_{2}$.
By repeating the argument, we can obtain the uniform bound of $\norm{\nabla u(t)}_{2}$ and reach
to the decay rate in \eqref{hln:rate} as given by \cite{HLN16}.

\subsection*{Notations}
For any $p \ge 1$, $L^p = L^{p}(\R^d)$ denotes the usual Lebesgue space on $\R^d$ equipped
with the norm $\norm{\, \cdot\, }_{p} \coloneqq \norm{\, \cdot\,}_{L^{p}}$.
Set $\ev{a}=(1+|a|^2)^{1/2}$ for $a \in \C$ or $\R^d$.
$\mathcal{F}[u] = \widehat{u}$ is the usual Fourier transform of a function $u$ on $\R^d$.
For $s \in \R$,
the standard Sobolev space on $\R^{d}$ is defined
by $H^{s} = H^{s}(\R^{d}) \coloneqq \{ u \in L^{2}(\mathbb{R}^d) \mid \norm{u}_{H^{s}} \coloneqq \norm{\ev{\nabla}^{s} u}_{L^2} < \infty \}$.
% $\mathcal{F}$ stands for the usual Fourier transform.
We denote the weighted Sobolev space by
$\mathcal{F}H^{s} = \mathcal{F}H^{s}(\R^{d}) \coloneqq \{ u \in L^{2}(\mathbb{R}^d)
\mid \norm{u}_{\mathcal{F}H^{s}} \coloneqq \norm{\ev{x}^{s}u}_{L^2} < \infty \}$.
Set $\Sigma = H^{1} \cap \mathcal{F} H^{1}$.
Let $U(t)$ be the Schr\"odinger group $e^{it \Delta/2}$.
$A \lesssim B$ denotes $A \le CB$ for some constants $C>0$.

\bigskip

This paper is organized as follows:
In Section \ref{sec:2}, we give some preliminary results.
The uniform-in-time estimates of $\norm{\nabla u(t)}_{2}$ and $\norm{x u(t)}_{2}$ are given
in Proposition \ref{prop:22} and Lemma \ref{lem:uni}, respectively, developing the energy-type estimate \eqref{ene:int}.
We also give the scale invariant inequality in Lemma \ref{lem:in1}.
Section \ref{sec:3} is devoted to the proof of Theorem \ref{thm:1}.
First, in Proposition \ref{prop:31}, we show the refined $L^{2}$ decay estimate compared to \cite{GKS-pre}.
Secondly, when $p < 1 + 4/(3d)$,
we improve the uniform-in-time estimate of $\norm{\nabla u(t)}_{2}$ in Proposition \ref{prop:33} by the iteration argument.
Finally, the same decay rate as \eqref{hln:rate}
is established at the end of Section \ref{sec:3}.
In Appendix \ref{app:1}, we justify the weighted energy-type estimate used in the proof of Lemma \ref{lem:uni}.

\section{Preliminary} \label{sec:2}

In this section, we give some results which play an important role throughout this paper.
Let us first recall Strichartz' estimates. A pair $(p,q)$ is said to be admissible if
\[
	2\le q,r \le \infty,\quad
      \frac{2}{q} = d\left(\frac12-\frac{1}{r} \right), \quad (d,q,r)\neq (2,2,\infty).
\]

\begin{lemma}[Strichartz' estimate, e.g., \cites{S77, GV85, Y87, KT98}] \label{lem:Str}
Let $(q,r)$ and $(\widetilde{q}, \widetilde{r})$ be admissible pairs.
For any interval $I \ni 0$,
\begin{align}
	\norm{U(t) f}_{L^{q}(I;L^{r}(\R^d))} \le{}& C_{0}\norm{ f}_{L^2(\R^d)}, \label{lem:Str1} \\
	\norm{\int_0^t U(t-s) F(s)\, ds}_{L^{q}(I;L^{r}(\R^d))}
	\le{}& C_{0} \norm{ F}_{L^{\widetilde{q}'}(I;L^{\widetilde{r}'}(\R^d))}, \label{lem:Str2}
\end{align}
where $C_{0} >0$ is a certain constant not depending on $I$, and $\rho'$ is the dual exponent of $\rho \ge 1$.
\end{lemma}

Let us define $J(t) = x+it \nabla$ for any $t \ge 0$, which is
connected to the Galilean symmetry of the equation \eqref{nls}.
% the generator of the Galilean transformation by
The operator possesses the property
\[
	J(t) = \mathcal{M}(t) it \nabla \mathcal{M}(-t) = U(t) x U(-t), \quad \mathcal{M}(t) \coloneqq e^{\frac{i|x|^2}{2t}}.
\]

Let us prove the existence of a unique global $H^{1}$-solution with $J(t) u \in L^{2}$
and the uniform-in-time estimate for $\norm{\nabla u(t)}_{2}$.
In particular, when $\Re \lambda <0$, the growth rate of the estimate is refined, compared to the previous work \cite{GKS-pre}.
\begin{proposition} \label{prop:22}
Let $1<p \le 1+4/d$ and $\Im \lambda < 0$.
Suppose $ u_{0} \in \Sigma$.
Then \eqref{nls} has a unique global solution $u \in C([0, \infty); H^{1})$ with $J(t) u \in C([0, \infty) ; L^{2})$. The solution exhibits
\begin{align}
	\norm{u(t)}_{2} \le{}& \norm{u_{0}}_{2} \label{uni:1}
\end{align}
for any $t \ge 0$.
Moreover, if $\Re \lambda \ge 0$, there exists $M(u_{0})>0$ not depending on $t$ such that the solution satisfies
\begin{align}
%	\norm{J(t) u(t)}_{2} \le{}& \norm{x u_{0}}_{2}, \label{uni:2} \\
	\norm{ \nabla u(t)}_{2} \le{}& M(u_{0}) \label{uni:3a}
\end{align}
for any $t \ge 0$.
% otherwise,
If $\Re \lambda < 0$,
the solution possesses
\begin{align}
	\norm{ \nabla u(t)}_{2} \lesssim{}& (1+t)^{\frac{1}{2}} \label{uni:3b}
\end{align}
for any $t \ge 0$.
\end{proposition}
\begin{proof}
The existence of the solution has been already proven by several authors, but we give a sketch of the proof for reader's convenience.
Let us introduce the complete metric space
\begin{align*}
	X_{T} \coloneqq{}& \left\{ u \in C([0, T] ; L^{2}) \cap L^{q_{0}}(0, T; L^{r_{0}})
\mid \norm{ u}_{X_{T}} \le 2C_{0} \norm{u_{0}}_{2} \right\}, \\
	\norm{ u}_{X_{T}} \coloneqq{}& \norm{ u}_{L^{ \infty}(0, T; L^{2})} + \norm{ u}_{L^{q_{0}}(0, T; L^{r_{0}})},
	\quad d( u, v) \coloneqq \norm{u-v}_{X_{T}},
%	\norm{ u - v}_{L^{ \infty}(0, T; L^{2})} + \norm{ u - v}_{L^{ q_{0}}(0, T; L^{r_{0}})},
\end{align*}
where $r_{0} = p+1$, $2/q_{0} = d/2 - d/r_{0}$ and $C_{0}>0$ is a certain constant given by Strichartz' estimate.
In what follows, we simply write $L^{q}(0, T; L^{r})$ as $L^{q}_{T}L^{r}$.
The standard well-posedness theory implies that there exist $T>0$ and a unique solution $u \in X_{T}$
to \eqref{nls} with $u \in L^{q}_{T} L^{r}$ for any admissible pair $(q,r)$ (for details, see \cite{YT87}).
Thanks to $\Im \lambda < 0$,
multiplying \eqref{nls} by $\overline{ u}$, taking the imaginary part in the both side and integrating over $\R^{d}$,
we reach to
\begin{align*}
	\frac{1}{2} \frac{d}{dt} \norm{u(t)}_{2}^{2} = \Im \lambda \norm{u(t)}_{p+1}^{p+1}\, dx
	\le 0,
\end{align*}
which yields $\norm{u(t)}_{2} \le \norm{u_{0}}_{2}$ for any $t \in [0, T]$.
Then the $L^{2}$-solution $u$ can be extended globally in time.
Further, since $p<1+4/d$, in view of $u_{0} \in \Sigma$, by means of $J(t) = \mathcal{M}(t) it \nabla \mathcal{M}(-t)$,
we easily verifies $\nabla u$, $J(\cdot) u \in C([0, \infty) ; L^{2}) \cap L^{q}_{loc}((0, \infty); L^{r})$
from the persistence of regularity of solutions via the standard bootstrap argument involving Strichartz' estimate.

Let us move on to prove \eqref{uni:3a} and \eqref{uni:3b}.
Multiplying \eqref{nls} by $\overline{ \partial_{t} u}$, taking the real part in the both side and integrating over $\R^{d}$,
we have
\begin{align*}
	\frac{d}{dt} E(u(t)) = 2 \lambda_{2} \int_{\R^{d}} |u(t)|^{p+1} \Im (u(t) \overline{\partial_{t} u(t)}) \, dx,
\end{align*}
where $ \lambda_{1} = \Re \lambda$, $ \lambda_{2} = \Im \lambda <0$ and
\[
	E( u(t)) = \frac{1}{2} \norm{ \nabla u(t)}_{2}^{2} + \frac{ \lambda_{1}}{p+1} \norm{u(t)}_{p+1}^{p+1}.
\]
Using \eqref{nls} again, one estimates
\begin{align}
\begin{aligned}
	\frac{d}{dt} E(u(t)) ={}& \frac{\lambda_{2}(p+1)}{2} \int_{\R^{d}} |u|^{p-1} | \nabla u|^{2}\, dx
+ \frac{\lambda_{2}(p-1)}{2} \int_{\R^{d}} |u|^{p-3} \Re \(u^{2} \overline{ \nabla u}^{2}\)\, dx \\
	&{}+ 2\lambda_{1}\lambda_{2} \norm{u}_{2p}^{2p} \\
	\le{}& \lambda_{2} \int_{\R^{d}} |u|^{p-1} |\nabla u|^{2}\, dx + 2\lambda_{1}\lambda_{2} \norm{u}_{2p}^{2p}.
\end{aligned}
	\label{en:ineq1}
\end{align}
When $ \lambda_{1} \ge 0$, \eqref{en:ineq1} yields $E(u(t)) \le E(u_{0})$ for any $t \ge 0$ and thus \eqref{uni:3a} is valid.
The case $ \lambda_{1} <0$ is more delicate.
Set $f = |u|^{\frac{p+1}{2}}$. Note that $\norm{ \nabla |u|} \le \norm{ \nabla u}$.
By the Gagliardo-Nirenberg and Young's inequality,
we see from \eqref{en:ineq1} that
\begin{align*}
	\frac{d}{dt} E(u(t)) \le{}& \frac{2\lambda_{2}}{p+1} \norm{ \nabla f}_{2}^{2}
	+ 2\lambda_{1}\lambda_{2}  \norm{f}_{\frac{4p}{p+1}}^{\frac{4p}{p+1}} \\
	\le{}& - \frac{4|\lambda_{2}|}{(p+1)^2}  \norm{ \nabla f}_{2}^{2}
	+ C \norm{u}_{2}^{2p\(1- \theta\)} \norm{ \nabla f}_{2}^{\frac{4p}{p+1} \theta} \\
	\le{}& %- \frac{|\lambda_{2}|}{(p+1)^2} \norm{ \nabla f}_{2}^{2}
	C \norm{u(t)}_{2}^{2p \kappa \(1- \theta\)},
\end{align*}
where
\[
	\theta = \frac{d(p-1)(p+1)}{p \( 4+ d(p-1) \)} \in (0,1), \quad
	\kappa = \frac{4+d(p-1)}{4-d(p-1)} >1,
\]
since $p < 1+4/d$.
This implies
\begin{align}
	\frac{d}{dt} E(u(t)) \lesssim \norm{u(t)}_{2}^{\frac{2(4p-d(p-1))}{4-d(p-1)}}.
	\label{en:ineq2}
\end{align}
Since $ \norm{u(t)}_{2} \le \norm{u_{0}}_{2}$, \eqref{en:ineq2} yields \eqref{uni:3b}.
This completes the proof.
\end{proof}

To prove the main result, we employ the following uniform-in-time estimate for $\norm{xu(t)}_{2}$:
\begin{lemma} \label{lem:uni}
Assume $1<p < 1+4/d$ and $\Im \lambda < 0$.
Let $ u$ be a global solution given by Proposition \ref{prop:22}.
Then the solution satisfies $u \in C([0, \infty) ; \mathcal{F}H^{1})$.
Moreover if $\Re \lambda \ge 0$, then the estimate
\[
	\norm{x u(t)}_{2} \lesssim 1+ t
\]
holds for any $t \ge 0$. If $\Re \lambda < 0$, the solutions satisfies
\[
	\norm{x u(t)}_{2} \lesssim (1+ t)^{\frac{3}{2}}
\]
for any $t \ge 0$.
\end{lemma}
\begin{remark}
When $\Re \lambda \ge 0$, the uniform estimate in Lemma \ref{lem:uni} was obtained in previous works \cite{K20a} and \cite{GKS24}
using the operator $J(t)=x+it\nabla$.
We directly establish the estimate
via a virial-type argument, which provides a simpler proof compared to the previous works.

\end{remark}
\begin{proof}
Noting $J(t) = x + it \nabla$, it follows that
\begin{align*}
	\norm{x (u(t) - u(t_{0}) )}_{2}
	\le{}& \norm{J(t) u(t) - J(t_{0}) u(t_{0})}_{2} + |t-t_{0}| \norm{\nabla u(t_{0})}_{2} + |t| \norm{u(t) - u(t_{0})}_{2}
	\rightarrow 0
\end{align*}
as $t \rightarrow t_{0}$, which implies $u \in C([0, \infty) ; \mathcal{F}H^{1})$.
Secondly, let us only treat the case $\Re \lambda <0$ because the other case is similar.
Multiplying \eqref{nls} by $|x|^{2} \overline{u}$, taking the imaginary part in the both side and integrating over $\R^{d}$, we have
\begin{align}
	\frac{1}{2} \frac{d}{dt} \norm{ x u(t)}_{2}^{2} \le \Im \int_{\R^{d}} \overline{u} x \cdot \nabla u\, dx.
	\label{eq:2}
\end{align}
Using \eqref{eq:2} and \eqref{uni:3b}, one deduces from H\"older's inequality that
\begin{align}
	\frac{d}{dt} \norm{x u(t)}_{2} \le \norm{ \nabla u(t)}_{2}
	\lesssim (1+t)^{\frac{1}{2}},
	\label{eq:3}
\end{align}
which implies the desired estimate. Thus the proof is completed.
\end{proof}

\begin{remark}
More precisely, we need the regularization argument to verify \eqref{en:ineq2} and \eqref{eq:3},
because the equation \eqref{nls} makes sense only in $H^{-1}$, and $\partial_{t} u$, $|x|^{2}u \not\in H^{1}$.
\eqref{en:ineq2} can be justified by the persistence of regularity of solutions, together with the standard density argument.
As for the justification of \eqref{eq:3}, we refer the reader to Appendix \ref{app:1}.
\end{remark}

The following scale-invariant inequality, which can be regarded as the dual of Gagliardo-Nirenberg's inequality, plays a crucial role in proving the main result.
The first author applied this inequality in \cite{K20a} to obtain an $L^2$-decay estimate for \eqref{nls}.
\begin{lemma}[Scale invariant inequality] \label{lem:in1}
Let $q \in [1, 2)$ if $d =1$ and $q \in (2d/(d+2), 2)$ if $d \ge 2$.
Set $\alpha = 1/q - 1/2$. Fix $m \in \N$.
Then it holds that
\[
	\norm{ f}_{q} \lesssim \norm{ f}_{2}^{1- \frac{d \alpha}{m}} \norm{|x|^{m} f}_{2}^{ \frac{d \alpha}{m}}.
\]
\end{lemma}
\begin{proof}
The case $\norm{f}_{2} = 0$ is obvious; otherwise, by the scaling and H\"older's inequality, we have
\begin{align*}
	\norm{ f}_{q} %={}& \lambda^{\frac{d}{q}} \norm{ f( \lambda \cdot)}_{q} \\
	\le{}& \lambda^{\frac{d}{q}} \norm{(1+ |y|^{m})^{-1}}_{2q/(2-q)} \norm{ (1+|y|^{m}) f( \lambda \cdot)}_{2}
	\lesssim \lambda^{d \alpha} \( \norm{ f}_{2} + \lambda^{-m} \norm{|x|^{m} f}_{2} \),
\end{align*}
from which the assertion follows, taking $ \lambda = \norm{|x|^{m} f}_{2}/ \norm{ f}_{2}$.
The proof is completed.
\end{proof}

\section{Proof of the main result} \label{sec:3}

As an initial step to prove Theorem \ref{thm:1}, we shall show the following $L^{2}$-decay estimate.
The proposition yields the result of the case $p=1+4/(3d)$ in Theorem \ref{thm:1}.

\begin{proposition} \label{prop:31}
Let $1< p \le 1+4/(3d)$.
Assume $\Re \lambda < 0$ and $\Im \lambda < 0$.
Given $u_{0} \in \Sigma$ with $ \norm{u_{0}}_{\Sigma} \neq 0$, let $u$ be a global solution given by Proposition \ref{prop:22}.
If $p = 1 + 4/(3d)$, then it holds that
\[
	\norm{ u(t)}_{2} \lesssim (\log (1+t))^{- \frac{3d}{2(d+2)}}
\]
for any $t \ge 1$.
When $1<p < 1+ 4/(3d)$, the estimate
\[
	\norm{ u(t)}_{2} \lesssim (1+t)^{- \frac{d}{d+2}\(\frac{2}{d(p-1)}- \frac{3}{2} \)}
\]
is valid for any $t \ge 0$.
\end{proposition}

\begin{proof}[Proof of Proposition \ref{prop:31}]
We shall follow the argument in \cites{K20a, GKS24}.
Multiplying \eqref{nls} by $\overline{ u}$, taking the imaginary part in the both side and integrating over $\R^{d}$,
we see from $\Im \lambda < 0$ and H\"older's inequality that
\begin{align*}
	\frac{1}{2} \frac{d}{dt} \norm{u(t)}_{2}^{2}
	= \Im \lambda \norm{ u(t)}_{p+1}^{p+1}
	\le -|\Im \lambda| \norm{u(t)}_{2}^{\beta} \norm{u(t)}_{q}^{p+1 - \beta},
\end{align*}
where $q =1$ if $d=1$ and $q = 2(d+1)/(d+2)$ if $d \ge 2$, and $\beta>p+1$ is defined by
\[
	\beta \alpha = \frac{p+1}{q} -1, \quad
	\alpha = \frac{1}{q} - \frac{1}{2}.
\]
By means of Lemma \ref{lem:in1}, thanks to Lemma \ref{lem:uni}, one has
\begin{align}
	\frac{1}{2} \frac{d}{dt} \norm{u(t)}_{2}^{2}
	\lesssim{}& - \norm{u(t)}_{2}^{p+1+ \frac{d(p-1)}{2}} \norm{x u(t)}_{2}^{\frac{d(1-p)}{2}}
%	\lesssim{}& - \frac{\norm{u(t)}_{2}^{\frac{(p+1)(d+2)-2d}{2}}}{ \( \norm{J(t)u(t)}_{2} + t \norm{\partial_{x} u(t)}_{2}\)^{\frac{d(p-1)}{2}} }
	\lesssim - (1+t)^{-\frac{3d(p-1)}{4}} \norm{u(t)}_{2}^{p+1+ \frac{d(p-1)}{2}},
	\label{eq:4}
\end{align}
which yields
\begin{align*}
	\frac{d}{dt} \( \norm{u(t)}_{2}^{-\frac{(d+2)(p-1)}{2}} \) \gtrsim (1+t)^{-\frac{3d(p-1)}{4}}
\end{align*}
for any $t \ge 0$.
Noting \eqref{uni:1} and $p \le 1+4/(3d)$, integrating the above for $t$, we establish the desired estimate.
\end{proof}

In the case $p<1+4/(3d)$, we will employ an iteration argument to update the decay rate of $\norm{u(t)}_{2}$ to the rate given by \cite{HLN16}.
Let us consider the recurrence relation
\begin{align}
	a_{n+1} = r a_{n} + q
	\label{rr:1}
\end{align}
for any $n \in \N$, where
\begin{align}
      \begin{aligned}
	r &= \frac{d\(4p-d(p-1)\)}{(d+2)\(4-d(p-1)\)} =1-\frac{2 \(4-3d(p-1) \)}{(d+2) \(4-d(p-1)\)}, \\
      q &= \frac{1}{2} - r \(\frac{2}{d(p-1)} - 1\)
            = - \frac{\( 4-3d(p-1) \) \( 4-(d-2) (p-1) \)}{2(d+2)(p-1) \( 4-d(p-1) \)}.
      \end{aligned}
      \label{rr:2}
\end{align}
The following is valid:
\begin{lemma} \label{lem:rr1}
Assume $1<p<1+4/(3d)$. Let $\{a_{n}\}_{n=1}^{ \infty}$ be a sequence satisfying \eqref{rr:1} with $a_{1} = 1/2$.
Then there exists $\alpha <0$ such that $\alpha = \lim_{n \rightarrow \infty} a_{n} $.
\end{lemma}
\begin{proof}
The relation \eqref{rr:1} rewrites $a_{n+1} - \alpha = r \( a_{n} - \alpha \)$ with $\alpha = q/(1-r)$.
Then $a_{n} = \alpha + r^{n-1} (a_{1} - \alpha)$ holds for any $n \in \N$.
Thanks to $1<p<1+4/(3d)$, we have $r \in (0,1)$ and $q<0$ from a direct calculation.
Hence it is concluded that $a_{n} \rightarrow \alpha <0$ as $n \rightarrow \infty$ as desired.
\end{proof}

With Lemma \ref{lem:rr1} in place, we can establish the time uniform bound of $ \norm{ \nabla u(t)}_{2}$.

\begin{proposition}[Iteration argument] \label{prop:33}
Let $1< p < 1+4/(3d)$.
Assume $\Re \lambda < 0$ and $\Im \lambda < 0$.
Take $u_{0} \in \Sigma$ with $ \norm{u_{0}}_{\ \Sigma} \neq 0$.
Let $u$ be a global solution given by Proposition \ref{prop:31}.
Then %there exists $M(u_{0})>0$ not depending on $t$ such that the solution satisfies
\eqref{uni:3a} holds for any $t \ge 0$.
\end{proposition}
\begin{proof}
%From \eqref{uni:3b}, the case $t \le 1$ is obvious.
%Let us deal with the case $t>1$.
We see from \eqref{uni:3b} and Proposition \ref{prop:31} that the solution satisfies
\[
	\norm{ u(t)}_{L^{2}} \lesssim (1+t)^{- \frac{d}{d+2}\(\frac{2}{d(p-1)}- (a_{1} + 1) \)}, \qquad
	\norm{ \nabla u(t)}_{2} \lesssim (1+t)^{a_{1}},
\]
%for any $t \ge 1$,
where $a_{1} = 1/2$.
Substituting the first estimate in the above into \eqref{en:ineq2}, one has
\begin{align*}
	\frac{d}{dt} E(u(t))
	\lesssim (1+t)^{- \frac{d}{d+2}\(\frac{2}{d(p-1)}- (a_{1} + 1) \)\frac{2(4p-d(p-1))}{4-d(p-1)}}.
\end{align*}
for any $t \ge 0$.
From a use of Gagliardo-Nirenberg and Young's inequality, this implies
\begin{align*}
	\norm{\nabla u(t)}_{2} \lesssim 1+ (1+t)^{r a_{1} + q} \lesssim (1+t)^{a_{2}},
\end{align*}
where $r$, $q$ are defined as in \eqref{rr:2} and $a_{2}$ fulfills \eqref{rr:1}.
%for any $t \ge 0$.
Plugging the above into \eqref{eq:3}, one gets
\begin{align*}
	\norm{x u(t)}_{2} \lesssim (1+ t)^{a_{2} +1}
\end{align*}
for any $t \ge 0$.
Hence it follows from \eqref{eq:4} that
\begin{align*}
	\frac{1}{2} \frac{d}{dt} \norm{u(t)}_{2}^{2}
	\lesssim{}& - \norm{u(t)}_{2}^{p+1+ \frac{d(p-1)}{2}} \norm{x u(t)}_{2}^{\frac{d(1-p)}{2}}
	\lesssim - (1+t)^{-\frac{d(p-1)}{2}(a_{2}+1)} \norm{u(t)}_{2}^{p+1+ \frac{d(p-1)}{2}},
\end{align*}
which yields
\begin{align*}
	\frac{d}{dt} \( \norm{u(t)}_{2}^{-\frac{(d+2)(p-1)}{2}} \) \gtrsim (1+t)^{-\frac{d(p-1)}{2}(a_{2}+1)}
\end{align*}
for any $t \ge 0$.
Integrating the above for $t$, we establish
\begin{align*}
	\norm{u(t)}_{L^{2}} \lesssim (1+t)^{- \frac{d}{d+2}\(\frac{2}{d(p-1)}- (a_{2} + 1)\)}.
\end{align*}
By iterating the above argument, we see from Lemma \ref{lem:rr1} that there exists $n \in \N$ such that $a_{n}<0$.
Hence, denoting the smallest $n$ such that $a_{n}<0$ by $n_{0}$, we establish
\begin{align*}
	\norm{\nabla u(t)}_{2} \lesssim M(u_{0}) - (1+t)^{a_{n_{0}}} \le M(u_{0}),
\end{align*}
as desired.
%for any $t \ge 1$.
\end{proof}

We are now in position to prove the case $p<1+4/(3d)$ in Theorem \ref{thm:1}.

\begin{proof}[Proof of the case $p<1+4/(3d)$ in Theorem \ref{thm:1}]
Let $u$ be a global solution given by Proposition \ref{prop:33}.
We see from Proposition \ref{prop:33} that there exists $M(u_{0})>0$ not depending on $t$ such that the solution satisfies
\begin{align*}
	\norm{ \nabla u(t)}_{2} \le{}& M(u_{0})
\end{align*}
for any $t \ge 0$.
Plugging the above into \eqref{eq:3}, $\norm{x u(t)}_{2} \lesssim 1+ t$ holds
for any $t \ge 0$.
Hence, similarly to the proof of Proposition \ref{prop:33}, we deduce from \eqref{eq:4} that
\begin{align*}
	\frac{1}{2} \frac{d}{dt} \norm{u(t)}_{2}^{2}
	\lesssim{}& - \norm{u(t)}_{2}^{p+1+ \frac{d(p-1)}{2}} \norm{x u(t)}_{2}^{\frac{d(1-p)}{2}}
	\lesssim - (1+t)^{-\frac{d(p-1)}{2}} \norm{u(t)}_{2}^{p+1+ \frac{d(p-1)}{2}},
\end{align*}
from which the desired estimate follows.
The proof is completed.

\end{proof}

\appendix

\section{Justification of the derivation to \eqref{eq:3} in Lemma \ref{lem:uni}} \label{app:1}
In this appendix, we shall justify the derivation of \eqref{eq:3} by the regularization argument (e.g., \cite[Proof of Lemma 6.5.2]{C03}).
Set
\[
	f_{\varepsilon}(t) = \norm{e^{-\varepsilon |x|^{2}}|x|u(t)}_{2}^{2}
\]
for any $t \ge 0$ and $\varepsilon>0$.
by the integration by part, we see from \eqref{nls} and $\Im \lambda < 0$ that
\begin{align*}
	f'_{\varepsilon}(t) &= \Im \int_{\R^{d}} \left\{ \nabla \( e^{-2\varepsilon |x|^{2}} |x|^{2} \overline{u(t)} \) \cdot \nabla u(t)  + 2 \lambda e^{-\varepsilon |x|^{2}} |x|^{2} |u(t)|^{p+1} \right\} dx \\
	&\le \Im \int_{\R^{d}} e^{-\varepsilon |x|^{2}}\(2- 4 \varepsilon |x|^{2} \) e^{-\varepsilon |x|^{2}} x  \overline{u(t)} \cdot \nabla u(t)\, dx.
\end{align*}
Since $e^{-\varepsilon |x|^{2}}\(2- 4 \varepsilon |x|^{2} \)$ is bounded in $x$ and $\varepsilon$, Cauchy-Schwarz' inequality gives us
\begin{align*}
	f'_{\varepsilon}(t) \le \norm{\nabla u(t)}_{2} \sqrt{f_{\varepsilon}(t)}.
\end{align*}
Integrating the above from $0$ to $t$, we obtain
\begin{align*}
	f_{\varepsilon}(t) \le \norm{x u_{0}}_{2} + \int_{0}^{t} \norm{\nabla u((s))}_{2} \sqrt{f_{\varepsilon}(s)}\, ds.
\end{align*}
Hence, Gronwall's inequality leads to
\begin{align*}
	\sqrt{f_{\varepsilon}(t)} \le \norm{xu_{0}}_{2} + C \int_{0}^{t} \norm{\nabla u(s)}_{2}\, ds.
\end{align*}
Thanks to Fatou's lemma, we have $\norm{x u(t)}_{2}^{2} \le f_{\varepsilon}(t)$ for any $t$. Then
\begin{align*}
	\norm{x u(t)}_{2} \le \norm{xu_{0}}_{2} + C \int_{0}^{t} \norm{\nabla u(s)}_{2}\, ds
\end{align*}
holds. The estimate is nothing but \eqref{eq:3}.
\hfill \qed

\subsection*{Acknowledgments}
N.K. was supported by JSPS KAKENHI Grant Numbers 23K03168 and 23K03183.
H.M. was supported by JSPS KAKENHI Grant Number 22K13941 and Kagawa University Research Promotion Program 2024 Grant Number 24K0D005.
T.S. was supported by JSPS KAKENHI Grant Numbers 22K13937 and 23KJ1765.



\begin{bibdiv}
\begin{biblist}

\bib{GA19}{book}{
      author={Agrawal, Govind~P.},
       title={Nonlinear {F}iber {O}ptics, {S}ixth {E}dition},
   publisher={Academic Press},
        date={2019},
        ISBN={978-0-12-817042-7},
         url={https://doi.org/10.1016/C2018-0-01168-8},
}

\bib{C03}{book}{
      author={Cazenave, Thierry},
       title={Semilinear {S}chr\"{o}dinger equations},
      series={Courant Lecture Notes in Mathematics},
   publisher={New York University, Courant Institute of Mathematical Sciences, New York; American Mathematical Society, Providence, RI},
        date={2003},
      volume={10},
        ISBN={0-8218-3399-5},
         url={https://doi.org/10.1090/cln/010},
      review={\MR{2002047}},
}

\bib{CHN-NA-2021}{article}{
      author={Cazenave, Thierry},
      author={Han, Zheng},
      author={Naumkin, Ivan},
       title={Asymptotic behavior for a dissipative nonlinear {S}chr\"{o}dinger equation},
        date={2021},
        ISSN={0362-546X},
     journal={Nonlinear Anal.},
      volume={205},
       pages={Paper No. 112243, 37},
         url={https://doi.org/10.1016/j.na.2020.112243},
      review={\MR{4212087}},
}

\bib{CMZ19}{article}{
      author={Cazenave, Thierry},
      author={Martel, Yvan},
      author={Zhao, Lifeng},
       title={Finite-time blowup for a {S}chr\"{o}dinger equation with nonlinear source term},
        date={2019},
        ISSN={1078-0947},
     journal={Discrete Contin. Dyn. Syst.},
      volume={39},
      number={2},
       pages={1171\ndash 1183},
         url={https://doi.org/10.3934/dcds.2019050},
      review={\MR{3918212}},
}

\bib{CN18}{article}{
      author={Cazenave, Thierry},
      author={Naumkin, Ivan},
       title={Modified scattering for the critical nonlinear {S}chr\"{o}dinger equation},
        date={2018},
        ISSN={0022-1236},
     journal={J. Funct. Anal.},
      volume={274},
      number={2},
       pages={402\ndash 432},
         url={https://doi.org/10.1016/j.jfa.2017.10.022},
      review={\MR{3724144}},
}

\bib{CW92}{article}{
      author={Cazenave, Thierry},
      author={Weissler, Fred~B.},
       title={Rapidly decaying solutions of the nonlinear {S}chr\"{o}dinger equation},
        date={1992},
        ISSN={0010-3616},
     journal={Comm. Math. Phys.},
      volume={147},
      number={1},
       pages={75\ndash 100},
         url={http://projecteuclid.org/euclid.cmp/1104250527},
      review={\MR{1171761}},
}

\bib{GKS-pre}{article}{
      author={Gerelmaa, Jadamba},
      author={Kita, Naoyasu},
      author={Sato, Takuya},
       title={{$L^2$}-decay estimate of solutions to dissipative nonlinear {S}chr\"odinger equations in {$\R^n$} without strong dissipative condition},
     journal={preprint},
}

\bib{GKS24}{article}{
      author={Gerelmaa, Jadamba},
      author={Kita, Naoyasu},
      author={Sato, Takuya},
       title={{$L^2$}-decay of solutions to dissipative nonlinear {S}chr\"{o}dinger equation with large initial data},
        date={2024},
        ISSN={1072-3374},
     journal={J. Math. Sci. (N.Y.)},
      volume={279},
      number={6},
       pages={814\ndash 823},
         url={https://doi.org/10.1007/s10958-024-07062-8},
      review={\MR{4731323}},
}

\bib{GOV94}{article}{
      author={Ginibre, J.},
      author={Ozawa, T.},
      author={Velo, G.},
       title={On the existence of the wave operators for a class of nonlinear {S}chr\"{o}dinger equations},
        date={1994},
        ISSN={0246-0211},
     journal={Ann. Inst. H. Poincar\'{e} Phys. Th\'{e}or.},
      volume={60},
      number={2},
       pages={211\ndash 239},
         url={http://www.numdam.org/item?id=AIHPA_1994__60_2_211_0},
      review={\MR{1270296}},
}

\bib{GV85}{article}{
      author={Ginibre, J.},
      author={Velo, G.},
       title={The global {C}auchy problem for the nonlinear {S}chr\"{o}dinger equation revisited},
        date={1985},
        ISSN={0294-1449},
     journal={Ann. Inst. H. Poincar\'{e} Anal. Non Lin\'{e}aire},
      volume={2},
      number={4},
       pages={309\ndash 327},
         url={http://www.numdam.org/item?id=AIHPC_1985__2_4_309_0},
      review={\MR{801582}},
}

\bib{HLN16}{article}{
      author={Hayashi, Nakao},
      author={Li, Chunhua},
      author={Naumkin, Pavel~I.},
       title={Time decay for nonlinear dissipative {S}chr\"{o}dinger equations in optical fields},
        date={2016},
        ISSN={1687-9120},
     journal={Adv. Math. Phys.},
       pages={Art. ID 3702738, 7},
         url={https://doi.org/10.1155/2016/3702738},
      review={\MR{3465033}},
}

\bib{HN-AJM-1998}{article}{
      author={Hayashi, Nakao},
      author={Naumkin, Pavel~I.},
       title={Asymptotics for large time of solutions to the nonlinear {S}chr\"{o}dinger and {H}artree equations},
        date={1998},
        ISSN={0002-9327},
     journal={Amer. J. Math.},
      volume={120},
      number={2},
       pages={369\ndash 389},
         url={http://muse.jhu.edu/journals/american_journal_of_mathematics/v120/120.2hayashi.pdf},
      review={\MR{1613646}},
}

\bib{KT98}{article}{
      author={Keel, Markus},
      author={Tao, Terence},
       title={Endpoint {S}trichartz estimates},
        date={1998},
        ISSN={0002-9327},
     journal={Amer. J. Math.},
      volume={120},
      number={5},
       pages={955\ndash 980},
         url={http://muse.jhu.edu/journals/american_journal_of_mathematics/v120/120.5keel.pdf},
      review={\MR{1646048}},
}

\bib{K20a}{article}{
      author={Kita, Naoyasu},
       title={Existence of blowing-up solutions to some {S}chr\"{o}dinger equations including nonlinear amplification with small data},
      %   date={2021},
     journal={J. appl. sci. eng., A, vol. 2, no. 1, pp. 5--10, Dec. 2021 (OCAMI Preprint Series 20-2 (2020))},
}

\bib{KS22b}{article}{
      author={Kita, Naoyasu},
      author={Sato, Takuya},
       title={Optimal {$L^2$}-decay of solutions to a nonlinear {S}chr\"{o}dinger equation with sub-critical dissipative nonlinearity},
        date={2022},
        ISSN={1021-9722},
     journal={NoDEA Nonlinear Differential Equations Appl.},
      volume={29},
      number={4},
       pages={Paper No. 41, 19},
         url={https://doi.org/10.1007/s00030-022-00772-5},
      review={\MR{4423544}},
}

\bib{KS-JDE-2007}{article}{
      author={Kita, Naoyasu},
      author={Shimomura, Akihiro},
       title={Asymptotic behavior of solutions to {S}chr\"{o}dinger equations with a subcritical dissipative nonlinearity},
        date={2007},
        ISSN={0022-0396},
     journal={J. Differential Equations},
      volume={242},
      number={1},
       pages={192\ndash 210},
         url={https://doi.org/10.1016/j.jde.2007.07.003},
      review={\MR{2361107}},
}

\bib{KS-JMSJ-2009}{article}{
      author={Kita, Naoyasu},
      author={Shimomura, Akihiro},
       title={Large time behavior of solutions to {S}chr\"{o}dinger equations with a dissipative nonlinearity for arbitrarily large initial data},
        date={2009},
        ISSN={0025-5645},
     journal={J. Math. Soc. Japan},
      volume={61},
      number={1},
       pages={39\ndash 64},
         url={http://projecteuclid.org/euclid.jmsj/1234189028},
      review={\MR{2272871}},
}

\bib{LP95}{article}{
      author={Liskevich, V.~A.},
      author={Perelmuter, M.~A.},
       title={Analyticity of submarkovian semigroups},
    language={English},
        date={1995},
        ISSN={0002-9939},
     journal={Proc. Am. Math. Soc.},
      volume={123},
      number={4},
       pages={1097\ndash 1104},
}

\bib{KS-AA-2021}{article}{
      author={Naoyasu, Kita},
      author={Takuya, Sato},
       title={Optimal {$L^2$}-decay of solutions to a cubic dissipative nonlinear {S}chr\"{o}dinger equation},
        date={2022},
        ISSN={0921-7134},
     journal={Asymptot. Anal.},
      volume={129},
      number={3-4},
       pages={505\ndash 517},
      review={\MR{4491888}},
}

\bib{OS-NoDEA-2020}{article}{
      author={Ogawa, Takayoshi},
      author={Sato, Takuya},
       title={{$\rm L^2$}-decay rate for the critical nonlinear {S}chr\"{o}dinger equation with a small smooth data},
        date={2020},
        ISSN={1021-9722},
     journal={NoDEA Nonlinear Differential Equations Appl.},
      volume={27},
      number={2},
       pages={Paper No. 18, 20},
         url={https://doi.org/10.1007/s00030-020-0621-3},
      review={\MR{4069852}},
}

\bib{OY02}{article}{
      author={Okazawa, Noboru},
      author={Yokota, Tomomi},
       title={Global existence and smoothing effect for the complex {G}inzburg-{L}andau equation with {$p$}-{L}aplacian},
        date={2002},
        ISSN={0022-0396},
     journal={J. Differential Equations},
      volume={182},
      number={2},
       pages={541\ndash 576},
         url={https://doi.org/10.1006/jdeq.2001.4097},
      review={\MR{1900334}},
}

\bib{S-AM-2020}{article}{
      author={Sato, Takuya},
       title={{$L^2$}-decay estimate for the dissipative nonlinear {S}chr\"{o}dinger equation in the {G}evrey class},
        date={2020},
        ISSN={0003-889X},
     journal={Arch. Math. (Basel)},
      volume={115},
      number={5},
       pages={575\ndash 588},
         url={https://doi.org/10.1007/s00013-020-01483-y},
      review={\MR{4154572}},
}

\bib{S22a}{article}{
      author={Sato, Takuya},
       title={Large time behavior of solutions to the critical dissipative nonlinear {S}chr\"{o}dinger equation with large data},
        date={2022},
        ISSN={1424-3199},
     journal={J. Evol. Equ.},
      volume={22},
      number={3},
       pages={Paper No. 59, 27},
         url={https://doi.org/10.1007/s00028-022-00815-5},
      review={\MR{4447337}},
}

\bib{S-CPDE-2006}{article}{
      author={Shimomura, Akihiro},
       title={Asymptotic behavior of solutions for {S}chr\"{o}dinger equations with dissipative nonlinearities},
        date={2006},
        ISSN={0360-5302},
     journal={Comm. Partial Differential Equations},
      volume={31},
      number={7-9},
       pages={1407\ndash 1423},
         url={https://doi.org/10.1080/03605300600910316},
      review={\MR{2254620}},
}

\bib{S77}{article}{
      author={Strichartz, Robert~S.},
       title={Restrictions of {F}ourier transforms to quadratic surfaces and decay of solutions of wave equations},
        date={1977},
        ISSN={0012-7094},
     journal={Duke Math. J.},
      volume={44},
      number={3},
       pages={705\ndash 714},
         url={http://projecteuclid.org/euclid.dmj/1077312392},
      review={\MR{512086}},
}

\bib{YT87}{article}{
      author={Tsutsumi, Yoshio},
       title={{$L^2$}-solutions for nonlinear {S}chr\"{o}dinger equations and nonlinear groups},
        date={1987},
        ISSN={0532-8721},
     journal={Funkcial. Ekvac.},
      volume={30},
      number={1},
       pages={115\ndash 125},
         url={http://www.math.kobe-u.ac.jp/~fe/xml/mr0915266.xml},
      review={\MR{915266}},
}

\bib{Y87}{article}{
      author={Yajima, Kenji},
       title={Existence of solutions for {S}chr\"{o}dinger evolution equations},
        date={1987},
        ISSN={0010-3616},
     journal={Comm. Math. Phys.},
      volume={110},
      number={3},
       pages={415\ndash 426},
         url={http://projecteuclid.org/euclid.cmp/1104159313},
      review={\MR{891945}},
}

\end{biblist}
\end{bibdiv}
\end{document}